\documentclass{amsart}
\usepackage{amssymb,xspace}
\usepackage{a4wide,tikz-cd}
\usepackage{bbm} 
\theoremstyle{plain}
\newtheorem{nthr}{Theorem}[section]
\newtheorem{theorem}		[nthr]{Theorem}
\newtheorem{proposition}[nthr]{Proposition}
\newtheorem{lemma}		[nthr]{Lemma}
\newtheorem{corollary}	[nthr]{Corollary}

\newtheorem{conjecture} [nthr]{Conjecture}

\theoremstyle{definition}

\def\mat#1{\ensuremath{#1}\xspace}
\def\dmat#1#2{\gdef#1{\mat{#2}}}
\def\oper#1#2{\dmat#1{\operatorname{#2}}}

\dmat\bk{\Bbbk}
\dmat\bA{\mathbb{A}}
\dmat\bC{\mathbb{C}}
\dmat\bF{\mathbb{F}}
\dmat\bG{\mathbb{G}}
\dmat\bH{\mathbb{H}}
\dmat\bL{\mathbb{L}}
\dmat\bN{\mathbb{N}}
\dmat\bQ{\mathbb{Q}}
\dmat\bP{\mathbb{P}}
\dmat\bR{\mathbb{R}}
\dmat\bT{\mathbb{T}}
\dmat\bZ{\mathbb{Z}}
\dmat\cA{\mathcal{A}}
\dmat\cC{\mathcal{C}}
\dmat\cF{\mathcal{F}}
\dmat\cE{\mathcal{E}}
\dmat\cH{\mathcal{H}}
\dmat\cM{\mathcal{M}}
\dmat\cO{\mathcal{O}}
\dmat\cP{\mathcal{P}}
\dmat\cS{\mathcal{S}}
\dmat\cT{\mathcal{T}}
\dmat\cU{\mathcal{U}}
\dmat\cV{\mathcal{V}}
\dmat\cZ{\mathcal{Z}}
\dmat\sA{\mathsf{A}}
\dmat\sH{\mathsf{H}}
\dmat\sI{\mathsf{I}}
\dmat\sJ{\mathsf{J}}
\dmat{\bfM}{\mathbf{M}}

\dmat\al\alpha
\dmat\be\beta
\dmat\ga\gamma
\dmat\de\delta
\dmat\eps{\varepsilon}
\dmat\hi\chi
\dmat\vi\varphi
\dmat\te\theta
\dmat\ka\kappa
\dmat\si\sigma
\dmat\Ga\Gamma
\dmat\Om\Omega
\dmat\ta{\tau}
\dmat\om{\omega}
\dmat\la{\lambda}

\oper\Hom{Hom}
\oper\Ext{Ext}
\oper\GL{GL}
\oper\End{End}
\oper\ch{ch}
\oper\cl{cl}
\oper\rk{rk}
\oper\Aut{Aut}
\oper\Exp{Exp}
\oper\Log{Log}
\oper\Coh{Coh}
\oper\Ind{Ind}
\oper\coker{coker}
\oper\im{im}
\oper\Pic{Pic}
\oper\Res{Res}

\def\oh{\frac12}
\def\iso{\simeq}
\def\ts{\otimes}
\def\lb#1{\mat{\underline{#1}}} 
\def\ub#1{\mat{\overline{#1}}}  
\def\wtl#1{\widetilde{#1}}
\def\br{\linebreak}
\def\mto{\mapsto}
\def\emb{\hookrightarrow}
\def\inv{^{-1}}
\def\bop{\bigoplus}

\def\nil{{\mathrm{nil}}}
\def\sst{{\mathrm{ss}}}
\def\bun{{\mathrm{vec}}}
\def\gen{{\mathrm{gen}}}
\def\coh{{\mathrm{coh}}}

\def\eq#1{\begin{equation}#1\end{equation}}
\def\eql#1#2{\begin{equation}\label{#2}#1\end{equation}}
\def\set#1{\mat{\{#1\}}}
\def\sets#1#2{\mat{\left\{#1\ \right\vert\left.#2\right\}}}
\def\ang#1{\mat{\left\langle#1\right\rangle}}
\def\rbr#1{\left(#1\right)}
\def\sbr#1{\left[#1\right]}
\def\n#1{\mat{\left\lvert#1\right\rvert}}

\def\sb{\subset}

\usepackage{hyperref}

\def\gnil{\mathbf{Nil}}

\def\gfil{\mathbf{Filt}}
\def\one{\mathbbm{1}}
\def\w{z}
\def\z{w}

\begin{document}
\title{Counting Higgs bundles}
\author{Sergey Mozgovoy}
\author[Olivier Schiffmann]{Olivier Schiffmann$^\dag$}
\email{mozgovoy@maths.tcd.ie}
\email{Olivier.Schiffmann@math.u-psud.fr}
\thanks{$^\dag$ partially supported by ANR grant 13-BS01-0001-01}


\begin{abstract}
We prove a closed formula counting semistable twisted Higgs bundles of fixed rank and degree over a smooth projective curve defined over a finite field.
We also prove a formula for the Donaldson-Thomas invariants of the moduli spaces of twisted Higgs bundles over a curve.
\end{abstract}

\maketitle


\section{Introduction}
The goal of this paper is to compute motivic invariants of the moduli spaces of twisted Higgs bundles over a curve.
Our approach is in part related to the approach of the second author \cite{schiffmann_indecomposable} where the case of usual Higgs bundles was considered, but it replaces the geometric deformation argument (only available in the symplectic case, and in high enough characteristic) by an argument involving the Hall algebra of the category of
(possibly twisted) Higgs pairs, which works in all characteristics.

Let $D$ be a divisor of degree $l$ over a smooth projective curve $X$ of genus $g$ over a field \bk.
A $D$-twisted Higgs bundle over $X$ is a pair $(E,\te)$, where $E$ is a vector bundle over $X$ and $\te\in\Hom(E,E(D))$.
A usual Higgs bundle is obtained when $D=K$, the canonical divisor of $X$.
One can define the notion of semistability for these pairs and one can construct 
the moduli stack $\cM_D(r,d)$ of semistable $D$-twisted Higgs bundles over $X$ of rank $r$ and degree $d$.
For $D=K$, it was conjectured by Hausel and Thaddeus \cite[Conj.3.2]{hausel_mirrora}
that motivic invariants $[\cM_D(r,d)]$ are independent of $d$ if $r$ and $d$ are coprime.
A conjectural formula for these invariants for $D=K$ was proposed by Hausel and Rodriguez-Villegas \cite[Conj.5.6]{hausel_mirrora} 
and for arbitrary $D$ with $\deg D\ge 2g-2$ by the first author \cite{mozgovoy_solutions} (see also \cite{chuang_motivic}).
In the case of $D=K$ this conjecture was verified for low ranks
\cite{garcia-prada_motives,gothen_betti}
and for the $y$-genus specialization \cite{garcia-prada_y}.
Some very interesting results for coprime $r$ and $d$ were also obtained in \cite{chaudouard_sura,chaudouard_sur}. There is a slightly more general conjecture about the Donaldson-Thomas invariants of twisted Higgs bundles.
We will denote by $q$ the motive $[\bA^1]$, called the Lefschetz motive (if we are working over a finite field then $q$ denotes the number of elements of this field).

\begin{conjecture}
Let $\sH_D(r,d)=(-q^\oh)^{-lr^2}[\cM_D(r,d)]$ and let the Donaldson-Thomas invariants
$\Om_D(r,d)$ be defined by
\eql{\sum_{d/r=\ta}\sH_D(r,d)\z^r\w^d=\Exp\rbr{\frac{\sum_{d/r=\ta}\Om_D(r,d)\z^r\w^d}{q-1}},\qquad \ta\in\bR.}{eq:Om0}
Then $\Om_D(r,d)$ are independent of $d$ ($d$ is not necessarily coprime to $r$ here)
and are given by the formula in \cite[Conj.3]{mozgovoy_solutions}.
\end{conjecture}

It is clear that if $r,d$ are coprime then $\Om_D(r,d)=(q-1)\sH_D(r,d)$.
It was proved in \cite[Th.1.5]{schiffmann_indecomposable} that if $D=K$
and $r,d$ are coprime then $(q-1)\sH_K(r,d)=q\sA(r,d)$, the number of absolutely
indecomposable vector bundles over $X$ of rank $r$ and degree $d$. We will see 
in Cor.~\ref{cor:Om=A}
that generally
\[\Om_K(r,d)=q\sA(r,d)\qquad \forall r,d.\]
This result is not surprising as a similar
statement for quiver representations was proved earlier by the first author \cite{mozgovoy_motivicb}. 

In this paper we will not prove the above conjecture, even the independence of $\Om_D(r,d)$ of $d$,
but we will give an explicit formula for these invariants.
Let $\cA^{>\ta}$ be the category of coherent sheaves over $X$ such that their HN-factors have slopes $>\ta$.
Let $\cA^{>\ta}_D$ be the category of $D$-twisted Higgs sheaves $(E,\te)$ with $E\in\cA^{>\ta}$. The categories
$\cA^{\geq \ta}, \cA^{\geq \ta}_D$ are defined in an analogous fashion.
We can define the notion of semistability for the objects in this category
and we can construct the moduli stacks $\cM^{>\ta}_D(r,d)$ of semistable bundles.
From now on we will use the index $+$ instead of $\ge0$
and write $\cA^+$ instead of $\cA^{\ge0}$ and so on. 
Note that $\cM^+_D(r,d)$ is not a substack of $\cM_D(r,d)$,
as not all objects in $\cM^+_D(r,d)$ are semistable in the usual sense.
However, we will show that if $d>\binom r2l$ then
$\cM^+_D(r,d)=\cM_D(r,d)$.

As $\cM_D(r,d)\iso\cM_D(r,d+r)$, it is enough to compute invariants $\cM^+_D(r,d)$ in order to determine all invariants $\cM_D(r,d)$.
It may seem that the above truncation makes things even more complicated.
But this is the only way to make all our counts finite 
and to be able to apply the standard machinery of wall-crossing (going back to Harder and Narasimhan)
to count semistable objects. More precisely, we will prove the following result (we assume that \bk is a finite field)

\begin{theorem}\label{main1}
Let $\cA^+_\bun(r,d)$ denote the set of isomorphism classes of vector bundles $E\in\cA^+$ with $\ch E=(r,d)$ and let
	\begin{align}
\sI^+_D(r,d)&=(-q^\oh)^{-lr^2}\sum_{E\in\cA^+_\bun(r,d)}\frac{[\Hom(E,E(D))]}{[\Aut E]},\\
\sum_{r,d}\sI^+_D(r,d)\z^r\w^d&=\Exp\rbr{\frac{\sum_{r,d}\Om^+_D(r,d)\z^r\w^d}{q-1}},\label{eq:Om1}\\
\sum_{d/r=\ta}\sH^+_D(r,d)\z^r\w^d
&=\Exp\rbr{\frac{\sum_{d/r=\ta}\Om^+_D(r,d)\z^r\w^d}{q-1}}\qquad\forall\ta>0.\label{eq:Om2}	
	\end{align}
If $l=\deg D\ge 2g-2$ then
\[\sH^+_D(r,d)=(-q^\oh)^{-lr^2}[\cM^+_D(r,d)].\]
If moreover $d>\binom r2l$ then
\[\sH_D(r,d)=\sH^+_D(r,d),\qquad \Om_D(r,d)=\Om_D^+(r,d).\]
\end{theorem}

This result implies that, for $\deg D\ge 2g-2$, in order to determine the invariants
$\sH_D(r,d)$ we need to compute invariants $\sH^+_D(r,d)$.
Denote by $\sI^+_{D,\nil}(r,d)$ analogues of $\sI^+_{D}(r,d)$ that count only
$\te\in\Hom(E,E(D))$ that are nilpotent (see section \ref{sec:nilp}). 
Similarly, we define invariants
$\sH^+_{D,\nil}(r,d)$ and $\Om^+_{D,\nil}(r,d)$.

\vspace{.1in}

The computation of $\Om_D(r,d)$ for $\deg D\ge 2g-2$ may be reduced to the nilpotent case via the following result~:

\begin{theorem}
We have
\begin{enumerate}
	\item $\Om^+_D(r,d)=\Om^+_{K-D,\nil}(r,d)$ if $\deg D>2g-2$.
	\item $\Om^+_K(r,d)=q\Om^+_{0,\nil}(r,d)$.
\end{enumerate}
\end{theorem}

The first statement of this theorem is proved in Corollary \ref{cr:D=K-D,nil} (cf.\ \cite[Cor.5.2.3]{chaudouard_sura}).
It is based on a simple observation (see Prop.\ \ref{D=K-D}) that 
\[\sI^+_D(r,d)=\sI^+_{K-D}(r,d)\]
and the fact that all elements of $\Hom(E,E(K-D))$ are nilpotent if $\deg D>2g-2$.
The second statement of this theorem is proved in Corollary \ref{cor:K0}.
It was proved in \cite[Cor.6.1.2]{chaudouard_sura}
and \cite[Th.1.5 and Prop.2.1]{schiffmann_indecomposable} for coprime $r$ and $d$.

\vspace{.1in}

Using techniques of \cite{schiffmann_indecomposable}, 
we will compute the invariants $\sI^+_{D,\nil}(r,d)$ for any divisor $D$ with $\deg D\le0$
-- and hence the invariants
$[\mathcal{M}_D(r,d)]$ for any divisor $D$ with
$\deg D \ge2g-2$.
In order to state our result, we need to introduce new notation.

\vspace{.1in}

Let 
$$Z_X(z)=\sum_{d \geq 0} [S^dX]z^d, \qquad \widetilde{Z}_X(z)=z^{1-g}Z_X(z)$$ 
denote the zeta function of $X$ and its renormalization. For a partition $\lambda=(1^{r_1}, 2^{r_2},\ldots,t^{r_t})$, we set
$$J_{\lambda}(\w)=\prod_{s \in \lambda} Z^*_X(q^{-1-l(s)}\w^{a(s)})$$
where $a(s)$ and $l(s)$ are respectively the arm and the leg lengths of $s\in\la$ \cite[VI.6.14]{macdonald_symmetric} and
\[
Z^*_X(q^{-1}z)
=\begin{cases}
Z_X(q^{-1}z)\qquad & \text{if } z\ne1,q,\\
\Res_{z=1}Z_X(q^{-1}z)=q^{1-g}\frac{[\Pic^0(X)]}{q-1}\qquad & \text{if }z=1.
\end{cases}
\]

Next, write $n=l(\lambda)=\sum_i r_i$,  
$$r_{<i}=\sum_{k<i} r_k, \qquad r_{>i}=\sum_{k>i}r_k, \qquad r_{[i,j]}=\sum_{k=i}^j r_k$$
and consider the rational function
$$L(\w_n, \ldots, \w_1)=\frac{1}{\prod_{i<j} \widetilde{Z}_X\big(\frac{\w_i}{\w_j}\big)} \sum_{\sigma \in \mathfrak{S}_n} \sigma \left[ \prod_{i<j}
\wtl Z_X\rbr{\frac{\w_i}{\w_j}} \cdot \frac{1}{\prod_{i<n} \rbr{1-q\frac{\w_{i+1}}{\w_i}}} \cdot 
\frac{1}{1-\w_1}\right].$$
Denote by $\Res_{\lambda}$ the operator of taking the iterated residue along
\begin{align*}
&\frac{\w_n}{\w_{n-1}}=\frac{\w_{n-1}}{\w_{n-2}}= \cdots = \frac{\w_{2+r_{<t}}}{\w_{1+r_{<t}}}=q^{-1}\\
&\vdots \qquad \qquad \vdots \qquad  \qquad\qquad \vdots\\
&\frac{\w_{r_1}}{\w_{r_1-1}}= \frac{\w_{r_1-1}}{\w_{r_1-2}}= \cdots = \frac{\w_{2}}{\w_{1}}=q^{-1}.
\end{align*}
Put
$$\widetilde{H}_{\lambda}(\w_{1+ r_{<t}}, \ldots, \w_{1+r_{<i}}, \ldots, \w_1)
=\Res_{\lambda}\left[ L(\w_n, \ldots, \w_1) \prod_{\substack{j =1 \\ j \not\in \{r_{\leq i}\}}}^{n}\frac{d\w_j}{\w_j} \right]$$
and finally
$$H_{\lambda}(\w)=\widetilde{H}_{\lambda}(\w^tq^{-r_{<t}}, \ldots, \w^iq^{-r_{<i}}, \ldots, \w).$$
Note that if $r_i=0$ for some $i$ then the function $\widetilde{H}_{\lambda}$ is independent of its $i$th argument.




\begin{theorem}\label{th:expl}
Let $D$ be a divisor on $X$ of degree $l\le0$. Then
\begin{equation}\label{E:formula_I_plus_D_Bun}
\sum_{r,d} \sI^{+}_{D,\nil}(r,d)\z^r\w^d
=\sum_\la(-q^\oh)^{(2g-2-l)\ang{\la,\la}}
J_\la(\w)H_\la(\w)\z^{\n\la},
\end{equation}
where $\ang{\la,\la}=\sum_{i\ge1}\la_i'\la_i'$ and $\la'$ is a partition conjugate to \la.
\end{theorem}

The above theorem is proved in section \ref{sec:nilp}.

\section{Twisted Higgs bundles}

Let $X$ be a curve of genus $g$ over a field \bk and let $D$ be a divisor over $X$ of degree $l$.
A $D$-twisted Higgs bundle is a pair $\ub E=(E,\te)$, where $E$ is vector bundle over $X$ and $\te\in\Hom(E,E(D))$.
Similarly, we define a $D$-twisted Higgs sheaf to be a pair $\ub E=(E,\te)$, where $E$ is a coherent sheaf over $X$ and $\te\in\Hom(E,E(D))$.
Let $\cA_D$ denote the category of $D$-twisted Higgs sheaves.

\subsection{Homological properties}
The category $\cA_D$ is an abelian category of homological dimension~$2$.
More precisely, we have

\begin{theorem}[Gothen-King {\cite{gothen_homological}}]
Given $\ub E=(E,\te)$, $\ub F=(F,\te')$ in $\cA_D$,
there is a long exact sequence
\begin{multline*}
0\to\Hom(\ub E,\ub F)\to\Hom(E,F)\to\Hom(E,F(D))\\
\to\Ext^1(\ub E,\ub F)\to\Ext^1(E,F)\to\Ext^1(E,F(D))
\to\Ext^2(\ub E,\ub F)\to0.
\end{multline*}
\end{theorem}

\begin{corollary}
For any $\ub E,\ub F\in\cA_D$, we have
\begin{enumerate}
	\item $\hi(\ub E,\ub F)=\hi(E,F)-\hi(E,F(D))=-l\rk E\rk F$.
	\item $\hi(\ub E,\ub F)=\hi(\ub F,\ub E)$.
\end{enumerate}
\end{corollary}

Applying Serre duality for coherent sheaves, we obtain the following form of Serre 
duality for Higgs bundles.

\begin{corollary}
For any $\ub E,\ub F\in\cA_D$, we have
\[\Ext^i(\ub E,\ub F)\iso\Ext^{2-i}(\ub F,\ub E(K-D))^*.\]
\end{corollary}

We define semistable objects in $\cA_D$ using the slope function
\[\mu(\ub E)=\frac{\deg E}{\rk E}.\]

\begin{corollary}\label{cr:sst zero}
Assume that $\deg D\ge 2g-2$ and $\ub E,\ub F\in\cA_D$ are semistable objects such that $\mu(\ub E)<\mu(\ub F)$.
Then $\Ext^2(\ub E,\ub F)=0$.
\end{corollary}
\begin{proof}
By our assumption $\deg(K-D)\le0$.
Therefore
\[\mu(\ub E(K-D))\le\mu(\ub E)<\mu(\ub F).\]
By the semistability of $\ub E(K-D)$ and $\ub F$ we conclude that
\[\Hom(\ub F,\ub E(K-D))=0.\]
This implies that $\Ext^2(\ub E,\ub F)=0$.
\end{proof}

\subsection{Hall algebra and quantum torus}
Let \cA be an abelian category over a finite field and $\cH$ be its Hall algebra.
Let $\Ga$ be a lattice equipped with a skew-symmetric form \ang{-,-}
and with a group homomorphism $\cl:K_0(\cA)\to\Ga$ such that
\[\hi(E,F)-\hi(F,E)=\ang{\cl E,\cl F},\qquad E,F\in\cA.\]
The algebra $\bT=\bQ(q^\oh)[\Ga]$ equipped with a product
\[e^\al\circ e^\be=(-q^\oh)^{\ang{\al,\be}}e^{\al+\be},\qquad \al,\be\in\Ga\]
is called the quantum affine torus. 
Both $\cH$ and $\bT$ are graded by the lattice $K_0(\cA)$. We will occasionally consider their completions
$$\prod_{\alpha \in K_0(\cA)} \cH[\alpha], \qquad \prod_{\alpha \in K_0(\cA)} \mathbb{T}[\alpha]$$
which we still denote by $\cH$ and $\bT$ respectively for simplicity.
We will denote completions of $\cH$ and $\bT$ by the same letters.
One defines the integration map
\eql{I:\cH\to\bT,\qquad [E]\mto (-q^\oh)^{\hi(E,E)}\frac{e^{\cl E}}{\n{\Aut E}}}{int map}
It is a ring homomorphism if \cA has homological dimension one \cite{reineke_counting}.
Generally it satisfies
\[I([E]\circ[F])=I(E)\circ I(F)\]
if $\Ext^k(F,E)=0$ for $k\ge2$. This explains the significance of Cor.\ \ref{cr:sst zero}.

\subsection{Invariants}\label{sec:invar}
Let us assume that the curve $X$ is defined over a finite field \bk.
We will denote the Hall algebra of $\cA_D$ by $\cH_D$. Define
\[\cl:K_0(\cA_D)\to\bZ^2,\qquad (E,\te)\mto(\rk E,\deg E).\]
As $\hi(\ub E,\ub F)=\hi(\ub F,\ub E)$ for any $\ub E,\ub F\in\cA_D$, we equip
$\bZ^2$ with a trivial skew-symmetric form. Therefore the quantum torus
$\bT=\bQ(q^\oh)\sbr{\bZ^2}$ is commutative.
As in the previous section, there is an integration map
\[I:\cH_D\to\bT,\qquad
[\ub E]\mto (-q^\oh)^{\hi(\ub E,\ub E)}\frac{e^{\cl \ub E}}{\n{\Aut \ub E}}.\]
We will use variables
\[\z=e^{(1,0)},\qquad \w=e^{(0,1)}.\]

Let $\bfM_D(r,d)$ be the set of isomorphism classes
of semistable $\ub E\in\cA_D$ with $\cl\ub E=(r,d)$.
Note that if $(E,\te)$ is semistable and has positive rank then $E$ is a vector bundle.
Define
\[\one^\sst_{r,d}=\sum_{\ub E\in\bfM_D(r,d)}[\ub E]\in\cH_D,\qquad
\sH_D(r,d)\z^r\w^d=I\rbr{\one^\sst_{r,d}}.\]
Then
\[\sH_D(r,d)=\sum_{\ub E\in\bfM_D(r,d)}(-q^\oh)^{\hi(\ub E,\ub E)}\frac{1}{\n{\Aut\ub E}}
=(-q^\oh)^{-lr^2}\sum_{\ub E\in\bfM_D(r,d)}\frac{1}{\n{\Aut\ub E}}
.\]
Note that $\sH_D(r,d)=\sH_D(r,d+r)$.

We define the DT invariants $\Om_D(r,d)$ of the twisted Higgs bundles by the formula
\[\sum_{d/r=\ta}\sH_D(r,d)\z^r\w^d
=\Exp\rbr{\sum_{d/r=\ta}\frac{\Om_D(r,d)}{q-1}\z^r\w^d},\qquad\ta\in\bR.\]
Comparing equations for $\ta$ and $\ta+1$, we obtain that $\Om(r,d)=\Om(r,d+r)$.
Various tests justify the conjecture  that $\Om(r,d)$ are independent of $d$ (cf.\ \cite[Conj.1.9]{chuang_motivic}).
Note that if $r,d$ are coprime then
\[\sH_D(r,d)=\frac{\Om_D(r,d)}{q-1}\]
and here the independence of $d$, in the case of $D=K_X$, was conjectured by
Hausel and Thaddeus \cite[Conj.3.2]{hausel_mirrora}.

\section{Positive Higgs bundles}
Let $X$ be a curve as before and let $\cA=\Coh X$ be the category of
coherent sheaves over $X$.
We define the notion of semistability on \cA using the slope function
$\mu(F)=\deg F/\rk F$.
For any $F\in\cA$, there exists a unique filtration
\[0=F_0\sb F_1\sb\dots\sb F_k=F\]
such that $G_i=F_i/F_{i-1}$ are semistable and
\[\mu(G_1)>\dots>\mu(G_k).\]
This filtration is called the Harder-Narasimhan filtration of $F$.
We define the spectrum $\si(F)\sb\bR\cup\set\infty$ of $F$ to be the set of all $\mu(G_i)$.
Given $\ta\in\bR$, we define the subcategories
\[\cA^{>\ta}=\sets{F\in\cA}{\inf\si(F)>\ta},\qquad
\cA^{<\ta}=\sets{F\in\cA}{\sup\si(F)<\ta}.
\]	
There is a torsion pair $(\cA^{>\ta},\cA^{\le\ta})$ of $\cA$ with $\cA^{>\ta}$ being a torsion part. Therefore $\cA^{>\ta}$ is closed under taking extensions and quotients in $\cA$.
Similarly there is a torsion pair 
$(\cA^{\ge\ta},\cA^{>\ta})$ with $\cA^{\ge\ta}$ being a torsion part.
We define the subcategory $\cA_D^{>\ta}\sb\cA_D$ to be the category of objects $(E,\te)\in\cA_D$ such that $E\in\cA^{>\ta}$.
We can define the notion of semistability on the category $\cA_D^{>\ta}$ using the same slope function as for $\cA_D$.
Let $\cA^+=\cA^{\ge0}$ and $\cA^+_D=\cA^{\ge0}_D$.
The objects of $\cA^+$ are called positive sheaves.
In the next statements we assume that $l=\deg D\ge0$.

\begin{lemma}\label{lm:sst gap}
	Assume that $E\in\cA$ and $\si(E)$ has a gap of length $>l$. Then there are no semistable pairs $(E,\te)\in\cA_D$.
\end{lemma}
\begin{proof}
	Let $a<b$ be the boundary points of the gap in $\si(E)$ of length $>l$. Then there exists an exact sequence
	\[0\to E'\to E\to E''\to0,\]
	where $E'\in\cA^{\ge b}$ and $E''\in\cA^{\le a}$. This implies that
	\[E''(D)\in\cA^{\le a+l}\sb\cA^{<b}\]
	and therefore $\Hom(E',E''(D))=0$. If $(E,\te)\in\cA_D$, then $\te(E')\sb E'(D)$ and therefore $(E',\te)$ is a destabilizing subobject of $(E,\te)$.
\end{proof}

\begin{lemma}
	\label{lm:pos}
	Assume that $E$ is a vector bundle of rank $r$ and degree $>\binom r2l$ and there exists a semistable pair $(E,\te)\in\cA_D$. Then $E\in\cA^{>0}$.
\end{lemma}
\begin{proof}
	Let $a_1<\dots<a_k$ be the elements of $\si(E)$ and $(r_1,\dots,r_k)$ be the ranks of the corresponding factors. By the previous result $a_i-a_{i-1}\le l$ and therefore $a_i\le a_1+(i-1)l$ and
	\[\frac{(r-1)l}2<\mu(E)
	=\frac{\sum a_ir_i}{r}
	\le\frac{\sum_{i=1}^k(a_1+(i-1)l)r_i}{r}
	\le \frac kr a_1+\frac{(r-1)l}{2}
	\]
	which implies $a_1>0$ and $E\in\cA^{>0}$.
	We used here the fact that if $\sum_{i=1}^kr_i=r$, then
	\[\sum_{i=1}^k(i-1)r_i\le\binom r2.\]
	Indeed, by induction
	\[\sum_{i=1}^k(i-1)r_i
	\le\max_{1\le r_1\le r}\rbr{\binom{r-r_1}{2}+(r-r_1)}
	=\max_{1\le r_1\le r}\binom{r-r_1+1}2=\binom r2.\]
\end{proof}

\begin{corollary}\label{cr:relation}
	Let $E$ be a vector bundle of rank $r$ and degree $>\binom r2l$. Then a pair $(E,\te)\in\cA_D$ is semistable if and only if it is semistable as an object of $\cA^{>0}_D$ (or $\cA^{\ge0}_D$).
\end{corollary}
\begin{proof}
	If $(E,\te)\in\cA_D$ is semistable then $E\in\cA^{>0}$ by Lemma \ref{lm:pos}. It is clear that $(E,\te)$ is semistable with respect to $\cA^{>0}_D$.
	
	Conversely, assume that $(E,\te)$ is semistable as an object of $\cA^{>0}_D$.
	If $(E,\te)$ is not semistable as an object of $\cA_D$, then it has a semistable destabilizing subobject $(F,\te')$ of rank $r'<r$. Then
	\[\mu(F)>\mu(E)>\frac{r-1}2l\ge\frac{r'-1}2l.\]
	By the previous result $F\in\cA^{>0}$ and therefore $(F,\te')$ is a destabilizing subobject of $(E,\te)$ in $\cA_D^{>0}$. This contradicts to the semistability of $(E,\te)$ in $\cA_D^{>0}$. The proof for $\cA^{\ge0}_D$ is analogous.
\end{proof}

\subsection{Comparison of invariants}

\begin{proof}[Proof of Theorem \ref{main1}]
Recall from section \ref{sec:invar}
that there is an integration map
\[I:\cH_D\to\bT,\] 
where $\cH_D$ is the (completed) Hall algebra of $\cA_D$ and \bT is the corresponding (completed) quantum torus.
Let $\one^{+,\bun}_{r,d}\in\cH_D$ be the sum over the objects $(E,\te)$ of
$\cA_D^+(r,d)$ such that $E$ is a vector bundle.
Let $\one^{+,\sst}_{r,d}\in\cH_D$ be the sum over the semistable objects $(E,\te)$ of
$\cA_D^+(r,d)$ (if $r>0$ then $E$ is a vector bundle). Then by uniqueness of Harder-Narasimhan filtrations, we have
\[\sum_{r,d}\one^{+,\bun}_{r,d}=\prod_{\ta\downarrow}
\rbr{\sum_{d/r=\ta}\one^{+,\sst}_{r,d}},\]
where the product is taken in the decreasing order of $\ta\in[0,+\infty)$.
If $\deg D\ge 2g-2$ then, by Cor.\ \ref{cr:sst zero}, the integration map preserves the product
on the right. If we define
\[\sH'_D(r,d)=(-q^\oh)^{-r^2l}[\cM^+_D(r,d)],\]
then, applying the integration map, we obtain
\[\sum_{r,d}\sI^+_D(r,d)\z^r\w^d=\prod_\ta
\rbr{\sum_{d/r=\ta}\sH'_D(r,d)\z^r\w^d}.\]
Note that the product on the right is not ordered anymore as the quantum torus is commutative
by \S\ref{sec:invar}.
Using the formula \eqref{eq:Om1} and comparing the slopes, we obtain
\[\sum_{d/r=\ta}\sH'_D(r,d)\z^r\w^d
=\Exp\rbr{\frac{\sum_{d/r=\ta}\Om^+_D(r,d)\z^r\w^d}{q-1}}\qquad\forall\ta>0.\]
This implies that $\sH'_D(r,d)=\sH^+_D(r,d)$.

Applying Cor.\ \ref{cr:relation}, we see that for $d>\binom r2l$, we have
\[\sH_D(r,d)=\sH'_D(r,d)=\sH^+_D(r,d).\]
To prove that $\Om_D(r,d)=\Om_D^+(r,d)$ for $d>\binom r2l$,
we compare the formulas \eqref{eq:Om0} and \eqref{eq:Om2}. We note
that if $d>\binom r2l$ and if $(r,d)=(kr',kd')$, for some $k\ge1$, then
\[kd'>\frac{kr'(kr'-1)}2l\ge\frac{k^2r'(r'-1)}2l\ge k\binom{r'}2l\]
and therefore $d'>\binom{r'}2l$. This implies that $\sH_D(r',d')=\sH^+_D(r',d')$
and we conclude that $\Om_D(r,d)=\Om_D^+(r,d)$.
\end{proof}

\begin{proposition}
\label{D=K-D}
We have
\[\Om_D^+(r,d)=\Om_{K-D}^+(r,d).\]
\end{proposition}
\begin{proof}
Let $h^i(E,F)=\dim\Ext^i(E,F)$. If $E$ is a vector bundle of rank $r$ then
\[h^0(E,E(D))=\hi(E,E(D))+h^0(E,E(K-D))
=r^2(1-g+l)+h^0(E,E(K-D)).\]
Therefore
\[-\oh r^2l+h^0(E,E(D))
=-\oh r^2\deg(K-D)+h^0(E,E(K-D)).\]
This implies that $\sI^+_D(r,d)=\sI^+_{K-D}(r,d)$.
By the definition of $\Om_D^+(r,d)$, we conclude that also
$\Om_D^+(r,d)=\Om_{K-D}^+(r,d)$.	
\end{proof}

\begin{corollary}\label{cr:D=K-D,nil}
If $\deg D>2g-2$ then
\[\Om_D^+(r,d)=\Om_{K-D,\nil}^+(r,d).\]
\end{corollary}
\begin{proof}
As $\deg(K-D)<0$, all morphisms in $\Hom(E,E(K-D))$ are nilpotent.
This implies that $\sI^+_{K-D}(r,d)=\sI^+_{K-D,\nil}(r,d)$.
Therefore also $\Om_{K-D}^+(r,d)=\Om_{K-D,\nil}^+(r,d)$ and we conclude that
$\Om_D^+(r,d)=\Om_{K-D}^+(r,d)=\Om_{K-D,\nil}^+(r,d)$.	
\end{proof}

In the next section we will prove a similar relation between $\Om_K^+(r,d)$ and $\Om_{0,\nil}^+(r,d)$.
\section{Indecomposable positive bundles}
The next two results are analogues of Lemmas \ref{lm:sst gap} and \ref{lm:pos}.
See also \cite[Prop.2.5]{schiffmann_indecomposable}.
 
\begin{lemma}
Let $E$ be an indecomposable vector bundle over $X$. Then $\si(E)$ does not have gaps
of length $>2g-2$.
\end{lemma}
\begin{proof}
	Let $a<b$ be the boundary points of the gap in $\si(E)$ of length $>2g-2$. Then there exists an exact sequence
	\[0\to E'\to E\to E''\to0,\]
	where $E'\in\cA^{\ge b}$ and $E''\in\cA^{\le a}$. This implies that
	\[E''(K)\in\cA^{\le a+2g-2}\sb\cA^{<b}\]
	and therefore $\Ext^1(E'',E')\iso\Hom(E',E''(K))^*=0$. 
	We conclude that the above sequence
	splits and $E$ is not indecomposable.	
\end{proof}

\begin{corollary}
	Assume that $E$ is an indecomposable vector bundle of rank $r$
	 and degree\br \mbox{$>\binom r2(2g-2)$}.
	Then $E\in\cA^{>0}$.
\end{corollary}
\begin{proof}
The proof is analogous to the proof of Lemma \ref{lm:pos}.
\end{proof}

Let $\sA^+(r,d)$ denote the number of positive (that is, contained in $\cA^{\ge0}$) absolutely indecomposable vector bundles of rank
$r$ and degree $d$.
The first formula of the next result was proved by the first author \cite{mozgovoy_motivicb}
in the case of quiver representations. The second formula was proved by the second author \cite{schiffmann_indecomposable}.
We give a unified approach based on \cite{mozgovoy_motivicb}.

\begin{theorem}
\label{th:Om_K}
We have
\begin{align}
\sum_{r,d}\sI^+_K(r,d)\z^r\w^d
&=\Exp\rbr{\frac{\sum_{r,d}\sA^+(r,d)\z^r\w^d}{1-q\inv}},\\
\sum_{r,d}\sI^+_{0,\nil}(r,d)\z^r\w^d
&=\Exp\rbr{\frac{\sum_{r,d}\sA^+(r,d)\z^r\w^d}{q-1}}.
\end{align}
\end{theorem}
\begin{proof}
To prove the first equation we apply the same approach as in
\cite[Theorem 5.1]{mozgovoy_motivicb}.
Let $\cA^+(r,d)$ denote the stack of positive vector bundles having rank $r$ and degree $d$ and
let $\cA^+_K(r,d)$ denote the stack of positive Higgs bundles.
Then the forgetful map
\[\cA^+_K(r,d)\to\cA^+(r,d)\]
has a fiber over $E\in\cA^+(r,d)$ that equals to
\[\Hom(E,E\ts K)\iso\Ext^1(E,E)^*.\]
If $E=\bop E_i^{n_i}$ is a decomposition of $E$ into the sum of indecomposable objects
then the contribution of $E$ to $[\cA^+_D(r,d)]$ is equal to (see \cite[Theorem 2.1]{mozgovoy_motivicb})
\[\frac{[\Hom(E,E\ts K)]}{[\Aut E]}
=\frac{[\Ext^1(E,E)]}{[\End(E)]\prod_i(q\inv)_{n_i}}
=\frac{q^{-\hi(E,E)}}{\prod_i(q\inv)_{n_i}},
\]
where $(q)_n=(1-q)\dots(1-q^n)$. Note that $\hi(E,E)=r^2(1-g)$.
We conclude from the proof of \cite[Theorem 5.1]{mozgovoy_motivicb} that
\begin{multline*}
\sum_{r,d}\sI^+_K(r,d)e^{(r,d)}
=\sum_{r,d}q^{r^2(1-g)}[\cA^+_K(r,d)]e^{(r,d)}\\
=\sum_{n:\Ind\to\bN}\prod_{E\in\Ind}\frac{e^{n(E)\ch E}}{(q\inv)_{n(E)}}
=\Exp\rbr{\frac{\sum \sA^+(r,d)e^{(r,d)}}{1-q\inv}}.
\end{multline*}

The proof of the second formula goes through the same lines.
Let $\cA^+_{0,\nil}(r,d)$ be the stack of nilpotent endomoprhisms and consider
the forgetful map
\[\cA^+_{0,\nil}(r,d)\to\cA^+(r,d).\]
If $E=\bop E_i^{n_i}$ is decomposed as before
then the contribution of $E$ in $[\cA^+_{0,\nil}(r,d)]$
is equal to \cite[Cor.2.4]{schiffmann_indecomposable}
\[\frac{[\Hom^\nil(E,E)]}{[\Aut E]}=\prod\frac{q^{-n_i}}{(q\inv)_{n_i}}.\]

Applying again the proof of \cite[Theorem 5.1]{mozgovoy_motivicb} we conclude that
\begin{multline*}
\sum_{r,d}\sI^+_{0,\nil}(r,d)e^{(r,d)}
=\sum_{r,d}[\cA^+_{0,\nil}(r,d)]e^{(r,d)}\\
=\sum_{n:\Ind\to\bN}\prod_{E\in\Ind}\frac{q^{-n(E)}e^{n(E)\ch E}}{(q\inv)_{n(E)}}
=\Exp\rbr{\frac{\sum q\inv\sA^+(r,d)e^{(r,d)}}{1-q\inv}}.
\end{multline*}
\end{proof}

\begin{corollary}
\label{cor:K0}
	We have
	\[\Om_K^+(r,d)=q\Om_{0,\nil}^+(r,d)=q\sA^+(r,d).\]
\end{corollary}

\begin{corollary}
\label{cor:Om=A}
	We have
	\[\Om_K(r,d)=q\sA(r,d).\]
\end{corollary}
\begin{proof}
If $d>\binom r2(2g-2)$ then $\Om_K(r,d)=\Om_K^+(r,d)$,
$\sA(r,d)=\sA^+(r,d)$ and we can apply the previous result.
For arbitrary $r,d$ we note that
$\Om_K(r,d)=\Om_K(r,d+r)$ and $\sA(r,d)=\sA(r,d+r)$.
\end{proof}

\section{Counting nilpotent pairs}
\label{sec:nilp}
Let $F$ be a coherent sheaf over $X$ and let $\te\in\Hom(F,F(D))$.
For any $k\ge1$, let $\te^k$ be the composition
\[F\to F(D)\to F(2D)\to\dots\to F(kD)\]
and let $F_k=\im\te^k(-kD)\subset F$.
We assume that $\te$ is nilpotent, that is, $\te^s=0$ for some $s\ge1$.
Then there is a chain of inclusions
\def\hlr{\hookleftarrow}
\def\epi{\twoheadrightarrow}
\[F=F_0\hlr F_1\hlr\dots \hlr F_s=0\]
and a chain of epimorphims
\[F=F_0\epi F_1(D)\epi\dots\epi F_s(sD)=0.\]

Define
\[
F'_k=\ker\rbr{F_k\xrightarrow\te F_{k+1}(D)},\qquad
F''_k=\coker\rbr{F_{k+1}\emb F_{k}}.\]
Then we have chains of monomorphims (resp.\ epimorphisms)
\eq{F'_0\hlr F'_1\hlr\dots \hlr F'_s=0,}
\eql{F''_0\epi F''_1(D)\epi\dots\epi F''_s(sD)=0.}{eq:filt}

The ranks and degrees of $F, F''_k, F''_k$ are related in the following fashion.
Let us put
$$\al_k=\ch F''_{k-1}((k-1)D)-\ch F''_k(kD),\qquad k=1,\dots,s$$
and write $\al_k=(r_k,d_k)$.
Then
\[\ch F''_k(kD)=\sum_{i>k}\al_i,\qquad
\ch F''_k=\sum_{i>k}\al_i(-kl),\]
where for $\al=(r,d)$ and $n\in\bZ$, we define $\al(n)=(r,d+nr)$. Therefore
\[\ch F_k=\sum_{i>k}\al_i(-kl)+\ch F_{k+1},
\qquad \rk F_k=\sum_{i>k}(i-k)r_i,\]
and
\begin{equation*}
\begin{split}
\ch F'_k&=\ch F_{k}-\ch F_{k+1}-(0,l\cdot \rk F_{k+1})\\
&=\sum_{i>k}\al_i(-kl)-\sum_{i>k+1}(0,l(i-k-1)r_i)\\
&=\sum_{i>k}(r_i,d_i-l(i-1)r_i)
=\sum_{i>k}\al_i((1-i)l).
\end{split}
\end{equation*}

\vspace{.1in}

Let now $\lb\al=(\al_1,\dots,\al_s)$ be any tuple of vectors in $\bZ^2$.
Let us denote by $\gfil(\lb\al)$ the stack of chains of epimorphisms of coherent sheaves
\[H_0\epi H_1\epi\dots\epi H_s=0\]
such that $\al_k=\ch\ker(H_{k-1}\to H_k)$. Let $\gnil_D(\lb\al)$ (resp. $\gnil(\lb\al)^+$) be the stack of pairs $(F,\te)\in\cA_D$, (resp. pairs $(F,\te)\in\cA^+_D$) such that the sequence \eqref{eq:filt} is in $\gfil(\lb\al)$.
We also set 
$$r_i = \rk \al_i, \qquad \lambda(\lb\al)=(1^{r_1}, 2^{r_2}, \ldots), \qquad r=\n{\lambda(\lb\al)}=\sum_i i r_i.$$

\begin{proposition}
The volume of every fiber of the map of stacks	
\[\gnil_D(\lb\al)\to\gfil(\lb\al),\qquad
(F,\te)\mto[F_0''\to F_1''(D)\to\dots \to F_s''(sd)],\]
is equal to
\[\prod_{k\ge0}q^{-\hi(F''_k,F'_{k+1})}.\]
\end{proposition}
\begin{proof}	
To reconstruct $(F,\te)$ we inductively build a diagram
\begin{equation}\label{d4}
\begin{tikzcd}
0\rar&F_{k+1}\rar[dashed]\dar{\te}&F_k\rar[dashed]\dar[dashed]&F_{k}''\rar\dar{\te}&0\\
0\rar&F_{k+2}(D)\rar&F_{k+1}(D)\rar&F_{k+1}''(D)\rar&0\\
&\cE'&\cE&\cE''
\end{tikzcd}
\end{equation}
for $k=s-1,\dots,1,0$.
Let us denote the columns of the diagram by $\cE',\cE,\cE''$. 
They are objects of the category of triples $\cE=(E_0,E_1,\te)$,
where $E_0,E_1$ are coherent sheaves over $X$ and $\te\in\Hom(E_0,E_1)$.
The above diagram can be interpreted as an exact sequence
\[0\to \cE'\to \cE\to \cE''\to0.\]
in the category of triples.
Applying the result of Gothen ang King \cite{gothen_homological}, we obtain the following
exact sequence
\begin{multline*}
0\to\Hom(\cE'',\cE')\to\bop_{i=0,1}\Hom(E''_i,E'_i)\to\Hom(E''_0,E'_1)\to\\
\to\Ext^1(\cE'',\cE')\to\bop_{i=0,1}\Ext^1(E''_i,E'_i)\to\Ext^1(E''_0,E'_1)
\to\Ext^2(\cE'',\cE')\to0.
\end{multline*}
Note that $\Ext^2(\cE'',\cE')=0$ as the map
\eql{\Ext^1(E''_0,E'_0)\to\Ext^1(E''_0,E'_1)}{es5}
is surjective by surjectivity of $E_0'\to E_1'$.
We are given already an element of $\Ext^1(E''_1,E'_1)$ and the possible lifts to 
$\Ext^1(\cE'',\cE')$ are parametrized by the kernel of \eqref{es5}.
Similarly to \cite[Lemma 3.3]{reineke_counting} or \cite[Cor.3.2]{garcia-prada_motives}
we conclude that the space of possible lifts has dimension
\[-\hi(\cE'',\cE')+\hi(E''_1,E'_1)
=\hi(E''_0,E'_1)-\hi(E''_0,E'_0)
=-\hi(F''_k,F'_{k+1}).
\]
\end{proof}

A direct computation yields
\begin{equation*}
\begin{split}
-\rho_l(\lb\al):&=\sum_{k\ge0}\hi(F''_k,F'_{k+1})\\
&=\sum_{k\ge0}\sum_{i>k}\sum_{j>k+1}\hi(\al_i(-kl),\al_j(1-j)l)\\
&=\sum_{k\ge0}\sum_{i>k}\sum_{j>k+1}\rbr{\hi(\al_i,\al_j)+r_ir_j(1+k-j)l}\\
&=\sum_{i\ge j>1}\rbr{(j-1)\hi(\al_i,\al_j)-\binom j2r_ir_jl}
+\sum_{j>i>0}\rbr{i\hi(\al_i,\al_j)+\rbr{\binom i2-i(j-1)}r_ir_jl}\\
&=\sum_{i<j}\rbr{i\hi(\al_i,\al_j)+(i-1)\hi(\al_j,\al_i)-i(j-1)r_ir_jl}
+\sum_i\rbr{(i-1)\hi(\al_i,\al_i)-\binom i2r_i^2l}.
\end{split}
\end{equation*}
Note that for $l=0$ this formula is equivalent to \cite[Prop.3.1, iii]{schiffmann_indecomposable}. Furthermore,
\begin{equation}\label{E:rhoalpha}
\rho_l(\lb\al)=\rho_0(\lb\al) +\frac{l}{2}\rbr{\sum_i ir_i}^2 -\frac{l}{2} \sum_{k \geq 1} \rbr{\sum_{i \geq k} r_i}^2
=\rho_0(\lb\al) + \frac{l}{2} r^2 -
\frac l2\ang{\lambda(\lb\al),\lambda(\lb\al)}.
\end{equation}

\vspace{.1in}

Let $\gfil^+(\lb\al)$ be the substack of $\gfil(\lb\al)$ consisting of epimorphisms of coherent sheaves
\[H_0\epi H_1\epi\dots\epi H_s=0\]
for which $H_0 \in \cA^+$. Similarly, let $\gnil^+_D(\lb\al)$ be the substack of $\gnil_D(\lb\al)$ corresponding to
pairs $(F,\te)$ for which $F \in \cA^+$.

\begin{lemma} Assume that $\deg(D) \leq 0$. Then the following diagram is cartesian~:
$$
\begin{tikzcd}
\gnil_D^+(\lb\al) \rar \dar &\gnil_D(\lb\al) \dar\\
\gfil^+(\lb\al) \rar  &\gfil(\lb\al)
\end{tikzcd}
$$
\end{lemma}
\begin{proof}
The statement of the lemma amounts to the following observations. Let $(F,\te)$ be an object of $\gnil_D(\lb\al)$, and let $F''_k$ be defined as above. If $F \in \cA^+$ then $F''_0\in \cA^+$. Conversely, if $F''_0\in\cA^+$, then
$F''_k(kD)\in\cA^+$ and therefore $F''_k\in\cA^+$, which in turn implies that $F\in\cA^+$.
\end{proof}

Let $\sI^{+,\coh}_{D,\nil}(\al)$ be an analogue of
$\sI^{+}_{D}(\al)$ counting pairs $(E,\te)$, where $E\in\cA^+$ is a coherent sheaf with $\ch E=\al$ and $\te\in\Hom(E,E(D))$ is nilpotent.

\begin{theorem}
We have
\begin{equation}\label{E:Isum}
\sum_{r,d} \sI^{+,\coh}_{D,\nil}(r,d)\z^r\w^{d}
=\sum_\la(-q^\oh)^{(2g-2-l)\ang{\la,\la}}
J_{\lambda}(\w)H_{\lambda}(\w) \z^{\n\la}
\cdot\Exp\rbr{\frac{[X]}{q-1} \cdot \frac{\w}{1-\w}}.
\end{equation}
\end{theorem}
\begin{proof}
Let $J(r,d)$ stand for the set of all tuples $\lb\al=(\al_1, \ldots, \al_s)$ such that $\sum i \al_i=(r,d)$ and $\al_s\ne0$, and let $J_\gen(r)$ stand for the set of all sequences $\lb r=(r_1, \ldots, r_t)$ such that $\sum_i ir_i=r$ and $r_t\geq 1$.
There is a natural map
$\pi:J(r,d) \to J_\gen(r)$ which assigns to a tuple
$(\al_1, \ldots, \al_s)$ the sequence $(\rk \al_1, \ldots, \rk \al_s)$ in which all the \textit{last} zero entries have been removed. Let us set 
$$\sI^{+,\coh}_{D,\nil}(\lb\al)=(-q^{\frac{1}{2}})^{-lr^2} [\gnil^+_{D}(\lb\al)]$$
so that, by (\ref{E:rhoalpha})
$$\sI^{+,\coh}_{D,\nil}(r,d)
=\sum_{\lb\al \in J(r,d)} \sI^{+,\coh}_{D,\nil}(\lb\al)
=(-q^\oh)^{-lr^2} \sum_{\lb\al \in J(r,d)} q^{\rho_0(\lb\al)+\frac l2 r^2-\frac l2\langle \lambda(\lb\al), \lambda(\lb\al)\rangle} [\gfil^+(\lb\al)].$$
Let us fix some $\lb r=(r_1, \ldots, r_t) \in J_\gen(r)$ and put $\lambda=(1^{r_1}, 2^{r_2}, \ldots,t^{r_t})$. Using \cite[Sec.5.6]{schiffmann_indecomposable} we have
\begin{equation*}
\begin{split}
\sum_{\lb\al \in \pi^{-1}(\lb r)} \sI^{+,\coh}_{D,\nil}(\lb\al)\w^{\sum_i i \deg \al_i}&
=(-q^{\frac{1}{2}})^{-l\langle \lambda,\lambda\rangle} \sum_{\al \in \pi^{-1}(\lb r)} q^{\rho_0(\lb\al)} [\gfil^+(\lb\al)] \w^{\sum_i i \deg \al_i}\\
&=(-q^{\frac{1}{2}})^{-l\langle \lambda,\lambda\rangle} q^{(g-1)\langle \lambda,\lambda\rangle} J_{\lambda}(\w)H_{\lambda}(\w) \cdot \Exp\left( \frac{[X]}{q-1} \cdot \frac{\w}{1-\w}\right)
\end{split}
\end{equation*}
Summing over $\lb r \in J_\gen(r)$ and over all $r$, we obtain formula \eqref{E:Isum}.
\end{proof}

\begin{proof}[Proof of Theorem~\ref{th:expl}]
Let $F \in \cA^+$ be a coherent sheaf, and $F=V \oplus T$ be a decomposition as a direct sum of a vector bundle $V$ and a torsion sheaf $T$. Observe that $V$ and $T$ belong to $\cA^+$. We have
$$\Hom(F,F(D)) = \Hom(V,V(D)) \oplus \Hom(T,T(D)) \oplus \Hom( V,T(D))$$
and $\te \in \Hom(F,F(D))$ is nilpotent if and only if its projections to $\Hom(V,V(D))$ and $\Hom(T,T(D))$ are. On the other hand there is a canonical exact sequence
\begin{equation*}
1\to\Hom(V,T)\to\Aut F\to\Aut V\times \Aut T\to1.
\end{equation*}
We deduce that
$$\frac{\n{\Hom^{\nil}(F,F(D))}}{\n{\Aut F}}=\frac{\n{\Hom^{\nil}(V,V(D))}}{\n{\Aut V}}\cdot \frac{\n{\Hom^{\nil}(T,T(D))}}{\n{\Aut T}}.$$
This implies
\begin{equation}
\sum_{r, d} \sI^{+,\coh}_{D, \nil}(r,d)\z^r\w^d
=\sum_{r,d}\sI^{+}_{D,\nil}(r,d)\z^r\w^d \cdot
\sum_{d}\sI^{+,\coh}_{D,\nil}(0,d)\w^d.
\label{eq:coh=bun*tor}
\end{equation}

Similarly to the second equality of Theorem~\ref{th:Om_K} (observe that the number of absolutely indecomposable torsion sheaves of degree $d>0$ is $\n{X(\bk)}$), we obtain
\[\sum_{d\ge0}\sI^{+,\coh}_{D,\nil}(0,d)\w^d
=\Exp\rbr{\frac{[X]}{q-1}\sum_{d\ge1}\w^d}
=\Exp\rbr{\frac{[X]}{q-1}\cdot\frac{\w}{1-\w}}.\]
This equation together with equations \eqref{E:Isum},
\eqref{eq:coh=bun*tor}
imply
$$
\sum_{r,d} \sI^{+}_{D,\nil}(r,d)\z^r\w^{d}
=\sum_\la(-q^\oh)^{(2g-2-l)\ang{\la,\la}}
J_{\lambda}(\w)H_{\lambda}(\w) \z^{\n\la}
.$$
\end{proof}



\providecommand{\bysame}{\leavevmode\hbox to3em{\hrulefill}\thinspace}
\providecommand{\href}[2]{#2}


\end{document}